\documentclass[12pt]{article}

\usepackage{datetime}

\usepackage[all]{xy}

\usepackage{amssymb, amsmath, amsthm,eucal}

\usepackage{mathptmx}
\usepackage[scaled=.90]{helvet}
\usepackage{courier}

\usepackage[pdftex]{hyperref}

\hypersetup{
  colorlinks = true,
  linkcolor = black,
  urlcolor = blue,
  citecolor = black,
}

\numberwithin{equation}{section}

\newcommand{\A}{\mathrm{A}}
\newcommand{\C}{\mathrm{C}}

\newcommand{\Br}{\mathrm{Br}}
\newcommand{\Pic}{\mathrm{Pic}}

\newcommand{\FF}{\mathbb F}
\newcommand{\QQ}{\mathbb Q}
\newcommand{\RR}{\mathbb R}
\newcommand{\ZZ}{\mathbb Z}

\newtheorem{proposition}{Proposition}[section]

\newtheorem{theorem}[proposition]{Theorem}
\newtheorem{corollary}[proposition]{Corollary}

\theoremstyle{definition}

\title{\bf The Brauer group of Burnside rings}

\author{Markus Szymik}

\date{March 2010}

\begin{document}

\maketitle

\begin{abstract}
\noindent
The Brauer group of a commutative ring is an important invariant of a commutative ring, a common journeyman to the group of units and the Picard group. Burnside rings of finite groups play an important r\^{o}le in representation theory, and their groups of units and Picard groups have been studied extensively. In this short note, we completely determine the Brauer groups of Burnside rings: they vanish.
\end{abstract}

\thispagestyle{empty}


\section*{Introduction}

Let~$G$ be a finite group. Disjoint union and cartesian product yield an addition and a multiplication on the set of isomorphism classes~$[S]$ of finite~$G$-sets~$S$. Except for the lack of additive inverses, these satisfy the axioms of a commutative ring. Formally adjoining additive inverses, one obtains the corresponding Grothendieck ring, the Burnside ring~$\A(G)$ of~$G$. See~\cite{Solomon} for the original source, and~\cite{Bouc:survey} for a recent comprehensive survey. Burnside rings, their modules, and algebras play an important r\^{o}le in representation theory and transformation groups. Consequently, many different algebraic properties of these rings have been studied, such as for example its idempotents, the prime ideal spectrum, and its representation type. See~\cite{Dress} and~\cite{Reichenbach} for these.

This short note determines another important invariant of the Burnside ring as a commutative ring: its Brauer group. This is a natural successor in the sequence of~K-theoretic invariants beginning with the group of units and the Picard group, which have been studied for Burnside rings extensively. We refer to \cite{Bass} for the general machinery of algebraic K-theory, but the few things used here will be recalled as needed. Thus, Section~\ref{sec:units} and Section~\ref{sec:Picard} are devoted to the units and the Picard group of commutative rings in general and of Burnside rings in particular. The Brauer group for commutative rings in general is recalled in Section~\ref{sec:Brauer}, while the final Section~\ref{sec:thm} contains the main result and its proof. The following Section~\ref{sec:ghost} reviews the most important means to study the Burnside ring of a finite group: the ghost map.


\section{The ghost map}\label{sec:ghost}

Arguably the most important means to study the Burnside ring~$\A(G)$ of a finite group~$G$ is its ghost ring~$\C(G)$. See~\cite{tomDieck:LNM} or \cite{Bouc:survey} for proofs of the results collected here. The ghost ring is the ring of integer valued functions on the set of conjugacy classes~$[H]$ of subgroups~$H$ of~$G$. For example, if~$S$ is finite~$G$-set, then~\hbox{$\Phi_S\colon[H]\mapsto|S^H|$}, the number of~$H$-fixed points of~$X$, is such a function. This induces a ring homomorphism
\begin{equation}\label{eq:ghost}
	\Phi\colon\A(G)\longrightarrow\C(G),\,S\longmapsto\Phi_S.
\end{equation}
The following observation goes back to Burnside, and is the reason for naming the rings~$\A(G)$ in his honor. 

\begin{proposition}
	The ghost map~\eqref{eq:ghost} is injective.
\end{proposition}

As the source and target of the ghost map have the same rank over~$\ZZ$, the cokernel is finite. In particular, there is an integer~$n$ such that 
\begin{equation}\label{eq:n}
	n\C(G)\subset\Phi(\A(G))\subset\C(G).
\end{equation}
In fact, we may take~$n=|G|$, the order of~$G$, and this shall be our choice from now on. We will also identify the Burnside ring~$\A(G)$ with its image~$\Phi(\A(G))$ in the ghost ring~$\C(G)$.


\section{Groups of units}\label{sec:units}

If~$A$ is a commutative ring, the group~$A^\times$ of its units is an abelian group under multiplication. For an obvious example, we have~$\ZZ^\times=\{\pm1\}$ by inspection. 

The groups of units~$\A(G)^\times$ of Burnside rings~$\A(G)$ are not yet fully understood. As~$\A(G)^\times$ is contained in~$\C(G)^\times$, and the units of the ghost ring are the functions with value~$\pm1$, it is clear that the group of units is an elementary abelian 2-group. Yoshida characterized which units of~$\C(G)$ lie in the Burnside ring, see~\cite{Yoshida}. Earlier, Matsuda has shown that, if~$G$ is abelian, the rank of the group of units equals~$1+s$, where~$s$ is number of subgroups of index 2, see~\cite{Matsuda}.

In~\cite{tomDieck:LNM}, tom Dieck has shown that if the order of~$G$ is odd, there are only trivial units~\hbox{$\A(G)^\times=\{\pm1\}$}. In fact, this result is shown to be equivalent to the Feit-Thompson theorem which states that every group of odd order is solvable. He also showed that~\hbox{$\A(G)^\times\not=\{\pm1\}$} if~$G$ is not solvable. 

For non-abelian, solvable groups of even order, the groups of units of the Burnside rings are still not completely understood. See for example \cite{Matsuda+Miyata}, \cite{Alawode}, \cite{Yalcin}, and~\cite{Bouc:units} for progress in this matter.


\section{Picard groups}\label{sec:Picard}

If~$A$ is a commutative ring, an~$A$-module~$M$ is invertible if there is another~$A$-module~$N$ such that~$M\otimes_AN\cong A$. The Picard group~$\Pic(A)$ is the set of isomorphisms classes~$[M]$ of invertible~$A$-modules~$M$ under tensor product. For example, we have~$\Pic(\ZZ)=0$ by the classification of finitely generated abelian groups. For later use we will record the following result.

\begin{proposition}\label{prop:finitePic}
	The Picard group of a finite commutative ring is trivial.
\end{proposition}

\begin{proof}
	A finite ring is a product of finitely many local rings. As there is a canonical isomorphism~\hbox{$\Pic(A\times B)\cong\Pic(A)\oplus\Pic(B)$}, it suffices to consider finite local rings. But the Picard group of any Noetherian local ring is trivial as a consequence of the Nakayama Lemma.
\end{proof}

The Picard groups of Burnside rings have been determined by tom Dieck in collaboration with Petrie. Using the notation of~\eqref{eq:n}, so that~$n=|G|$, the order of~$G$, the fundamental result reads as follows.

\begin{proposition}
	The Picard group of the Burnside ring~$\A(G)$ is isomorphic to the cokernel of the map
	\begin{displaymath}
		\C(G)^\times\oplus(\A(G)/n\C(G))^\times\longrightarrow(\C(G)/n\C(G))^\times
	\end{displaymath}
	induced by the canonical homomorphisms.
\end{proposition}

The kernel of this map is~$\A(G)^\times$. See~\cite{tomDieck+Petrie},~\cite{tomDieck:Manu},~\cite{tomDieck:LNM},~\cite{tomDieck:Aarhus}, and~\cite{tomDieck:Crelle} for the full story. 


\section{Brauer groups}\label{sec:Brauer}

The Brauer group has its origin in the study of central division algebras over number fields by Brauer, Hasse, and Noether. The definition has been generalized by Azumaya to local rings, and by Auslander and Goldman to general commutative rings. See~\cite{Auslander+Goldman} for the original source and~\cite{Orzech+Small} for a useful set of lecture notes. 

Briefly, elements of the Brauer group of a commutative ring~$A$ are represented by isomorphism classes of central separable~$A$-algebras~$D$. The tensor product induces the structure of an abelian monoid on these, with~$A$ as the neutral element. The Brauer group is obtained by modding out the central separable~$A$-algebras of the form~$\mathrm{End}_A(M)$, where~$M$ is a finitely generated projective faithful~$A$-module. Since~$D\otimes_AD^\circ\cong\mathrm{End}_A(D)$ for central separable~$A$-algebras~$D$, where~$D^\circ$ is~$D$ with the opposite multiplication, this is indeed a group. 

Wedderburn's theorem, see~\cite{Wedderburn}, implies that the Brauer group of a finite field is trivial. In fact, the following more general result is Corollary~5.9 in~\cite{Orzech+Small}.

\begin{proposition}\label{prop:finiteBr}
	The Brauer group of a finite commutative ring is trivial.
\end{proposition}

As another example, we have~$\Br(\ZZ)=0$. In contrast to the results about the group of units and the Picard group of~$\ZZ$, this is non-trivial to see, and is usually deduced from class field theory. (See~\cite{Morris} for a direct attempt.) For lack of a reference,~I will include a proof here.

\begin{proposition}\label{prop:Z}
	The Brauer group of the ring of integers is trivial.
\end{proposition}

\begin{proof}
	As~$\ZZ$ is a regular domain, there is an injection of~$\Br(\ZZ)$ into~$\Br(\QQ)$, see Theorem 7.2 in~\cite{Auslander+Goldman}. Hasse, Brauer, and Noether's work \cite{Brauer+Hasse+Noether} yielded a local-global principle:~$\Br(\QQ)$ injects into
	\begin{displaymath}
		\Br(\RR)\oplus\bigoplus_p\Br(\QQ_p)\cong\ZZ/2\oplus\bigoplus_p\QQ/\ZZ,
	\end{displaymath}
	and the image consists of those families such that their sum (in~$\QQ/\ZZ$) vanishes. However, the image in the summands~$\Br(\QQ_p)$ is trivial since the map from~$\Br(\ZZ)$ to~$\Br(\QQ_p)$ factors through~$\Br(\ZZ_p)$, which is complete local, so~$\Br(\ZZ_p)\cong\Br(\FF_p)$ by Azumaya's Theorem, see Theorem 6.5 in~\cite{Auslander+Goldman}, and~$\Br(\FF_p)=0$ by Wedderburn's theorem above. The image in~$\Br(\RR)$ must also be zero as otherwise the sum would not be zero.
\end{proof}

Note that the Picard groups show that the non-triviality of an invariant of commutative rings for~$\ZZ$ does not imply its non-triviality for Burnside rings. Therefore, the main Theorem~\ref{thm:main} below is non-obvious and needs proof.


\section{The main theorem}\label{sec:thm}

\begin{theorem}\label{thm:main}
	The Brauer group of a Burnside ring vanishes.
\end{theorem}

\begin{proof}
	Using again the notation from~\eqref{eq:n}, the Burnside ring~$\A(G)$ can be written as a pullback
	\begin{center}
  \mbox{ 
    \xymatrix{
    	\A(G)\ar[r]\ar[d]&
	\C(G)\ar[d]\\
	\A(G)/n\C(G)\ar[r]&
	\C(G)/n\C(G)
    } 
  }
\end{center}
	of rings. Clearly, the map~$\C(G)\rightarrow\C(G)/n\C(G)$ is surjective. 
	
	If, more generally, there is a pullback 
	\begin{center}
  	\mbox{ 
    \xymatrix{
    	A\ar[r]\ar[d]&
	C\ar[d]\\
	A'\ar[r]&
	C'
    } 
  	}
	\end{center}
	of rings with~$C\rightarrow C'$ surjective, one might expect an exact sequence
	\begin{displaymath}
		\Pic(C')\overset{\partial}{\longrightarrow}\Br(A)\longrightarrow
		\Br(A')\oplus\Br(C)\longrightarrow\Br(C'),
	\end{displaymath}
	the Mayer-Vietoris sequence for the Brauer group, where the unlabeled arrows are the canonical maps.  Unfortunately, in general, this exists only under additional hypotheses, see Theorem~4.1 in~\cite{Childs} and the stronger Theorem~2.2 in~\cite{Knus+Ojanguren}. In fact, already the definition of~$\partial$ is established only under further assumptions. Fortunately, in our situation, these are all satisfied:
	
		By Proposition~\ref{prop:finitePic}, we have~$\Pic(\C(G)/n\C(G))=0$. This implies that~$\partial$ is defined and that the sequence is exact at~$\Br(\A(G))$, so that the group of interest injects into~$\Br(\A(G)/n\C(G))\oplus\Br(\C(G))$.
	
	The ring~$\A(G)/n\C(G)$ is finite. Hence its Brauer group vanishes by Proposition~\ref{prop:finiteBr}.
	
	The ghost ring~$\C(G)$ is a product of copies of~$\ZZ$, one for each conjugacy class of subgroups of~$G$. If a ring splits as a product of two subrings, the Brauer group of it splits as well into the Brauer groups of the subrings. See~\cite{Orzech+Small}, p.~58, for example. Therefore, an induction on the number of conjugacy classes of subgroups, together with Proposition~\ref{prop:Z}, shows that~$\Br(\C(G))=0$ as well.
	
	To sum up, the Brauer group~$\Br(\A(G))$ is a subgroup of the trivial group.
\end{proof}

\begin{corollary}
	Every central separable algebra over a Burnside ring is isomorphic to the endomorphism algebra of a 	finitely generated projective faithful module.
\end{corollary}



\vfill

\parbox{\linewidth}{%
Markus Szymik\\
Department of Mathematical Sciences\\
NTNU Norwegian University of Science and Technology\\
7491 Trondheim\\
NORWAY\\
\href{mailto:markus.szymik@ntnu.no}{markus.szymik@ntnu.no}\\
\href{https://folk.ntnu.no/markussz}{folk.ntnu.no/markussz}}

\end{document}